\documentclass[preprint,12pt]{elsarticle}

\usepackage{amsmath,mathrsfs}
\usepackage{amssymb}

\usepackage{color}
\usepackage[all]{xy}
\usepackage{graphicx,xspace,bm,subfigure}
\usepackage{url}
\usepackage{setspace}

\newtheorem{theorem}{Theorem}[section]

\newtheorem{lemma}[theorem]{Lemma}

\newtheorem{remark}[theorem]{Remark}
\newtheorem{example}[theorem]{Example}

\newtheorem{proposition}[theorem]{Proposition}  
\newproof{proof}{Proof}

\newcommand{\norm}[1]{\ensuremath{\| #1 \|}}

\newcommand{\real}{{\mathbb{R}}}
\newcommand{\realpositive}{{\mathbb{R}}_{>0}}
\newcommand{\realnonnegative}{{\mathbb{R}}_{\ge 0}}

\newcommand{\integerspositive}{\mathbb{Z}_{\geq 1}}

\newcommand{\eps}{\epsilon}
\newcommand{\argmin}{\operatorname{argmin}}
\newcommand{\until}[1]{\{1,\dots,#1\}}
\newcommand{\map}[3]{#1:#2 \rightarrow #3}
\newcommand{\setmap}[3]{#1:#2 \rightrightarrows #3}

\newcommand{\Lie}{\mathcal{L}}

\newcommand{\gradient}{\nabla}

\newcommand{\zeros}{\mathbf{0}}
\newcommand{\setdef}[2]{\{#1 \; | \; #2\}}

\renewcommand{\SS}{\mathcal{S}}
\newcommand{\WW}{\mathcal{W}}

\newcommand{\DD}{\real^n \times \realnonnegative^m}

\newcommand{\KK}{\mathcal{K}}

\newcommand{\intr}{\mathrm{int}}

\newcommand{\cl}{\mathrm{cl}}
\newcommand{\bd}{\mathrm{bd}}

\newcommand{\xo}{x_{*}}

\newcommand{\xt}{\tilde{x}}
\newcommand{\lmt}{\tilde{\lambda}}

\newcommand{\Xp}{X_{\text{p-d}}}

\newcommand{\proj}{\mathrm{proj}}

\newcommand{\lm}{\lambda}
\newcommand{\lmo}{\lambda_{*}}

\newcommand{\sX}{\mathsf{X}}
\newcommand{\slm}{\Lambda}

\newcommand{\oprocendsymbol}{\hbox{$\bullet$}}
\newcommand{\oprocend}{\relax\ifmmode\else\unskip\hfill\fi\oprocendsymbol}

\newcommand{\longthmtitle}[1]{\mbox{}\textup{\textsl{(#1):}}}

\newcommand{\myclearpage}{\clearpage}
\renewcommand{\myclearpage}{}

\begin{document}

\begin{frontmatter}

  \title{Asymptotic convergence\\
    of constrained primal-dual dynamics}

  
  \author[ucsd]{Ashish Cherukuri}\ead{acheruku@ucsd.edu} 
  \author[caltech]{Enrique Mallada}\ead{mallada@caltech.edu}
  \author[ucsd]{Jorge Cort\'{e}s}\ead{cortes@ucsd.edu}

  \address[ucsd]{Department of Mechanical and Aerospace Engineering,
    University of California, San Diego, CA 92093, USA}
  \address[caltech]{Department of Computational and Mathematical
    Sciences, California Institute of Technology, Pasadena, CA 91125,
    USA}

\begin{abstract}
  This paper studies the asymptotic convergence properties of the
  primal-dual dynamics designed for solving constrained concave
  optimization problems using classical notions from stability
  analysis.  We motivate the need for this study by providing an
  example that rules out the possibility of employing the invariance
  principle for hybrid automata to study asymptotic convergence.  We
  understand the solutions of the primal-dual dynamics in the
  Caratheodory sense and characterize their existence, uniqueness, and
  continuity with respect to the initial condition.  We use the
  invariance principle for discontinuous Caratheodory systems to
  establish that the primal-dual optimizers are globally
  asymptotically stable under the primal-dual dynamics and that each
  solution of the dynamics converges to an optimizer.
\end{abstract}


\end{frontmatter}

\section{Introduction}\label{se:Intro}

The (constrained) primal-dual dynamics is a widespread continuous-time
algorithm for determining the primal and dual solutions of an
inequality constrained convex (or concave) optimization problem.  This
dynamics, first introduced in the pioneering
works~\cite{KA-LH-HU:58,TK:56}, has been used in multiple
applications, including network resource allocation problems for
wireless systems~\cite{DF-FP:10,JC-VKNL:12,AF-FP:14} and distributed
stabilization and optimization of power
networks~\cite{XM-NE:13,CZ-UT-NL-SL:14,EM-CZ-SL:14,XZ-AP:14}. 


Our objective in this paper is to provide a rigorous treatment of the
convergence analysis of the primal-dual dynamics using classical
notions from stability analysis. Since this dynamics has a
discontinuous right-hand side, the standard Lyapunov or LaSalle-based
stability results for nonlinear systems, see e.g.~\cite{HKK:02}, are
not directly applicable. This observation is at the basis of the
direct approach to establish convergence taken in~\cite{KA-LH-HU:58},
where the evolution of the distance of the solution of the primal-dual
dynamics to an arbitrary primal-dual optimizer is approximated using
power series expansions and its monotonic evolution is concluded by
analyzing the local behavior around a saddle point of the terms in the
series. 
Instead,~\cite{DF-FP:10} takes an indirect approach to establish
convergence, modeling the primal-dual dynamics as a hybrid automaton
as defined in~\cite{JL-KHJ-SNS-JZ-SSS:03}, and invoking a generalized
LaSalle Invariance Principle to establish asymptotic
convergence. However, the hybrid automaton that corresponds to the
primal-dual dynamics is in general not continuous, thereby not
satisfying a key requirement of the invariance principle stated
in~\cite{JL-KHJ-SNS-JZ-SSS:03}, and invalidating this route to
establish convergence. The first contribution of this paper is an
example that illustrates this point. Our second contribution is an
alternative proof strategy to arrive at the same convergence results
of~\cite{DF-FP:10}.  

For the problem setup, we consider an inequality constrained concave
optimization problem described by continuously differentiable
functions with locally Lipschitz gradients.  Since the primal-dual
dynamics has a discontinuous right-hand side, we specify the notion of
solution in the Caratheodory sense (note that this does not
necessarily preclude the study of other notions of solution).
We show that the primal-dual dynamics is a particular case of a
projected dynamical system and, using results from~\cite{AN-DZ:96}, we
establish that Caratheodory solutions exist, are unique, and are
continuous with respect to the initial condition. Using these
properties, we show that the omega-limit set of any solution of the
primal-dual dynamics is invariant under the dynamics. Finally, we
employ the invariance principle for Caratheodory solutions of
discontinuous dynamical systems from~\cite{AB-FC:06} to show that the
primal-dual optimizers are globally asymptotically stable under the
primal-dual dynamics and that each solution of the dynamics converges
to an optimizer.  We believe the use of classical notions of stability
and Lyapunov methods provides a conceptually simple and versatile
approach that can also be invoked in characterizing other properties
of the dynamics.

The paper is organized as follows. Section~\ref{sec:prelims} presents
basic notation and preliminary notions on discontinuous dynamical
systems. Section~\ref{sec:problem} introduces the primal-dual dynamics
and motivates with an example the need for a convergence analysis with
classical stability tools. Section~\ref{sec:convergence} presents the
main convergence results. Finally, Section~\ref{sec:conclusions}
gathers our conclusions and ideas for future work.

\myclearpage
\section{Preliminaries}\label{sec:prelims}

This section introduces notation and basic concepts about
discontinuous and projected dynamical systems.

\subsection{Notation}

We let $\real$, $\realnonnegative$, $\realpositive$, and
$\integerspositive$ be the set of real, nonnegative real, positive
real, and positive integer numbers, respectively.  We denote by
$\norm{\cdot}$ the $2$-norm on $\real^n$.  The open ball of radius
$\delta > 0$ centered at $x \in \real^n$ is represented by
$B_{\delta}(x)$. Given $x\in \real^n$, $x_i$ denotes the $i$-th
component of $x$. For $x,y \in \real^n$, $x \le y$ if and only if $x_i
\le y_i$ for all $i \in \until{n}$. We use the shorthand notation
$\zeros_n = (0,\ldots,0) \in \real^n$.  For a real-valued function
$\map{V}{\real^n}{\real}$ and $\alpha > 0$, we denote the sublevel set
of $V$ by $V^{-1}(\le \alpha) = \setdef{x \in \real^n}{V(x) \le
  \alpha}$. For scalars $a, b \in \real$, the operator $[a]_{b}^+$ is
defined as
\begin{align*}
  [a]_{b}^+ = \begin{cases} a, & \quad \text{ if } b> 0,
    \\
    \max\{0,a\}, & \quad \text{ if } b = 0.
  \end{cases}
\end{align*}
For vectors $a,b \in \real^n$, $[a]_{b}^+$ denotes the vector whose
$i$-th component is $[a_i]_{b_i}^+$, $i \in \until{n}$. For a set $\SS
\in \real^n$, its interior, closure, and boundary are denoted by
$\intr(\SS)$, $\cl(\SS)$, and $\bd(\SS)$, respectively. 
Given two sets $X$ and $Y$, a set-valued map $\setmap{f}{X}{Y}$
associates to each point in $X$ a subset of $Y$.  A map
$\map{f}{\real^n}{\real^m}$ is locally Lipschitz at $x \in \real^n$ if
there exist $\delta_x, L_x > 0$ such that $\norm{f(y_1) - f(y_2)} \le
L_x \norm{y_1 - y_2}$ for any $y_1, y_2 \in B_{\delta_x}(x)$. If $f$
is locally Lipschitz at every $x \in \KK \subset \real^n$, then we
simply say that $f$ is locally Lipschitz on~$\KK$.  The map $f$ is
Lipschitz on $\KK \subset \real^n$ if there exists a constant $L > 0$
such that $\norm{f(x) - f(y)} \le L \norm{x-y}$ for any $x,y \in
\KK$. Note that if $f$ is locally Lipschitz on $\real^n$, then it is
Lipschitz on every compact set $\KK \subset \real^n$.  The map $f$ is
locally bounded if for each $x \in \real^n$ there exists constants
$M_x, \epsilon_x>0$ such that $\norm{f(y)} \le M_x$ for all $y \in
B_{\epsilon_x}(x)$.

\subsection{Discontinuous dynamical systems}\label{subsec:disc}

Here we present basic concepts on discontinuous dynamical systems
following~\cite{AB-FC:06,JC:08-csm-yo}. Let
$\map{f}{\real^n}{\real^n}$ be Lebesgue measurable and locally bounded
and consider the differential equation
\begin{equation}\label{eq:dis-dyn}
  \dot x = f(x) .
\end{equation}
A map $\map{\gamma}{[0,T)}{\real^n}$ is a \emph{(Caratheodory)
  solution} of~\eqref{eq:dis-dyn} on the interval $[0,T)$ if it is
absolutely continuous on $[0,T)$ and satisfies $\dot \gamma(t) =
f(\gamma(t))$ almost everywhere in $[0,T)$.
A set $\SS \subset \real^n$ is \emph{invariant}
under~\eqref{eq:dis-dyn} if every solution starting from any point in
$\SS$ remains in $\SS$.  For a solution $\gamma$ of~\eqref{eq:dis-dyn}
defined on the time interval $[0,\infty)$, the \emph{omega-limit} set
$\Omega(\gamma)$ is defined by
\begin{align*}
  \Omega(\gamma) = \setdef{y \in \real^n}{\text{$\exists
      \{t_k\}_{k=1}^{\infty} \subset [0,\infty)$ with $\lim_{k \to
        \infty} t_k = \infty$ and $\lim_{k \to \infty} \gamma(t_k) =
      y$}} .
\end{align*}
If the solution $\gamma$ is bounded, then $\Omega (\gamma) \neq
\emptyset$ by the Bolzano-Weierstrass theorem~\cite{WR:53}.  These
notions allow us to characterize the asymptotic convergence properties
of the solutions of~\eqref{eq:dis-dyn} via invariance principles.
Given a continuously differentiable function
$\map{V}{\real^n}{\real}$, the \emph{Lie derivative of $V$
  along~\eqref{eq:dis-dyn}} at $x \in \real^n$ is $\Lie_f V(x) =
\gradient V(x)^\top f(x)$.  The next result is a simplified version
of~\cite[Proposition 3]{AB-FC:06} which is sufficient for our
convergence analysis later.

\begin{proposition}\longthmtitle{Invariance principle for
    discontinuous Caratheodory systems}\label{pr:invariance-cara}
  Let $\SS \in \real^n$ be compact and invariant. Assume that, for
  each point $x_0 \in \SS$, there exists a unique solution
  of~\eqref{eq:dis-dyn} starting at $x_0$ and that its omega-limit set
  is invariant too. Let $\map{V}{\real^n}{\real}$ be a continuously
  differentiable map such that $\Lie_f V(x) \le 0$ for all $x \in
  \SS$. Then, any solution of~\eqref{eq:dis-dyn} starting at $\SS$
  converges to the largest invariant set in $\cl(\setdef{x \in
    \SS}{\Lie_f V(x) = 0})$.
\end{proposition}

\subsection{Projected dynamical systems}\label{subsec:projected}

Projected dynamical systems are a particular class of discontinuous
dynamical systems. Here, following~\cite{AN-DZ:96}, we gather some
basic notions that will be useful later to establish continuity with
respect to the initial condition of the solutions of the primal-dual
dynamics. Let $\KK \subset \real^n$ be a closed convex set.  Given a
point $y \in \real^n$, the (point) projection of $y$ onto $\KK$ is
$\proj_{\KK}(y) = \argmin_{z \in \KK} \norm{z - y}$. Note that
$\proj_{\KK}(y)$ is a singleton and the map $\proj_{\KK}$ is Lipschitz
on $\real^n$ with constant $L = 1$~\cite[Proposition 2.4.1]{FHC:83}.
Given $x \in \KK$ and $v \in \real^n$, the (vector) projection of $v$
at $x$ with respect to~$\KK$~is
\begin{equation*}
  \Pi_{\KK}(x,v) = \lim_{\delta \to 0^+} \frac{\proj_{\KK}(x+\delta v) -
    x}{\delta} .
\end{equation*}
Given a vector field $\map{f}{\real^n}{\real^n}$ and a closed convex
polyhedron $\KK \subset \real^n$, the associated projected dynamical
system is
\begin{equation}\label{eq:pds}
  \dot x = \Pi_{\KK}(x,f(x)), \quad x(0) \in \KK,
\end{equation}
Note that, at any point $x$ in the interior of $\KK$, we have
$\Pi_{\KK}(x,f(x)) = f(x)$.  At any boundary point of $\KK$, the
projection operator restricts the flow of the vector field $f$ such
that the solutions of~\eqref{eq:pds} remain in $\KK$.  Therefore, in
general,~\eqref{eq:pds} is a discontinuous dynamical system.  The next
result summarizes conditions under which the (Caratheodory) solutions
of the projected system~\eqref{eq:pds} exist, are unique, and
continuous with respect to the initial condition.

\begin{proposition}\longthmtitle{Existence, uniqueness, and continuity
    with respect to the initial condition~\cite[Theorem
    2.5]{AN-DZ:96}}\label{pr:existence-pds} 
  Let $\map{f}{\real^n}{\real^n}$ be Lipschitz on a closed convex
  polyhedron $\KK \subset \real^n$. Then,
  \begin{enumerate}
  \item (existence and uniqueness): for any $x_0 \in \KK$, there
    exists a unique solution $t \mapsto x(t)$ of the projected
    system~\eqref{eq:pds} with $x(0) = x_0$ defined over the domain
    $[0,\infty)$,
  \item (continuity with respect to the initial condition): given a
    sequence of points $\{x_k\}_{k=1}^{\infty} \subset \KK$ with
    $\lim_{k \to \infty} x_k= x$, the sequence of
    solutions $\{t \mapsto \gamma_k(t)\}_{k=1}^\infty$ of~\eqref{eq:pds}
    with $\gamma_k(0) = x_k$ for all $k$, converge
    to the solution $t \mapsto \gamma(t)$ of~\eqref{eq:pds} with
    $\gamma(0)= x$ uniformly on every compact set of $[0,\infty)$.
  \end{enumerate}
\end{proposition}

\myclearpage
\section{Problem statement}\label{sec:problem}

This section reviews the primal-dual dynamics for solving constrained
optimization problems and justifies the need to rigorously
characterize its convergence properties.  Consider the concave
optimization problem on~$\real^n$,
\begin{subequations}\label{eq:concave-opt}
  \begin{align}
    \mathrm{maximize} \quad & f(x),
    \\
    \text{subject to} \quad & g(x) \le \zeros_m, \label{eq:inequality}
  \end{align}
\end{subequations}
where the continuously differentiable functions
$\map{f}{\real^n}{\real}$ and $\map{g}{\real^n}{\real^m}$ are strictly
concave and convex, respectively, and have locally Lipschitz gradients. 
The Lagrangian of the problem~\eqref{eq:concave-opt} is given as
\begin{equation}\label{eq:lagrangian}
  L(x,\lambda) = f(x) - \lambda^\top g(x),
\end{equation}
where $\lambda \in \real^m$ is the Lagrange multiplier corresponding
to the inequality constraint~\eqref{eq:inequality}.  Note that the
Lagrangian is concave in $x$ and convex (in fact linear) in $\lm$.
Assume that the Slater's conditions is satisfied for the
problem~\eqref{eq:concave-opt}, that is, there exists $x \in \real^n$
such that $g(x) < \zeros_m$.  Under this assumption, the duality gap
between the primal and dual optimizers is zero and a point $(\xo,\lmo)
\in \DD$ is a primal-dual optimizer of~\eqref{eq:concave-opt} if and
only if it is a saddle point of $L$ over the domain $\DD$,~i.e.,
\begin{align*}
  L(x,\lm) \le L(\xo,\lmo) \quad \text{and} \quad L(\xo,\lm) \ge
  L(\xo,\lmo) ,
\end{align*}
for all $x \in \real^n$ and $\lm \in \realnonnegative^m$.  For
convenience, we denote the set of saddle points of $L$ (equivalently
the primal-dual optimizers) by $\sX \times \slm \subset \real^n \times
\real^m$. Note that since $f$ is strictly concave, the set $\sX$ is a
singleton.  Furthermore, $(\xo,\lmo)$ is a primal-dual optimizer if
and only if it satisfies the following Karush-Kuhn-Tucker (KKT)
conditions (cf.~\cite[Chapter 5]{SB-LV:09}),
\begin{subequations}\label{eq:KKT}
  \begin{align}
    \gradient f(\xo) - \sum_{i=1}^m (\lmo)_i \gradient g_i(\xo) & = 0,
    \\
    g(\xo) \le \zeros_m, \quad \lmo \ge \zeros_m, \quad \lmo^\top
    g(\xo) & = 0.
  \end{align}
\end{subequations}
Given this characterization of the solutions of the optimization
problem, it is natural to consider the \emph{primal-dual dynamics} on
$\real^n \times \realnonnegative^m$ to find them
\begin{subequations}\label{eq:p-d-dynamics}
  \begin{align}
    \dot x & = \gradient_x L(x,\lambda) = \gradient f(x) -
    \sum_{i=1}^m \lambda_i \gradient g_i(x) ,
    \\
    \dot \lambda & = [-\gradient_{\lambda} L(x,\lambda)]_{\lambda}^+ =
    [g(x)]_{\lambda}^+ .
  \end{align}
\end{subequations}
When convenient, we use the notation $\map{\Xp}{\real^n \times
  \realnonnegative^m}{\real^n \times \real^m}$ to refer to the
dynamics~\eqref{eq:p-d-dynamics}.  Given that the primal-dual dynamics
is discontinuous, we consider solutions in the Caratheodory sense. The
reason for this is that, with this notion of solution, a point is an
equilibrium of~\eqref{eq:p-d-dynamics} if and only if it satisfies the
KKT conditions~\eqref{eq:KKT}.


Our objective is to establish that the solutions
of~\eqref{eq:p-d-dynamics} exist and asymptotically converge to a
solution of the concave optimization problem~\eqref{eq:concave-opt}
using classical notions and tools from stability analysis.  Our
motivation for this aim comes from the conceptual simplicity and
versatility of Lyapunov-like methods and their amenability for
performing robustness analysis and studying generalizations of the
dynamics.  One way of tackling this problem, see
e.g.,~\cite{DF-FP:10}, is to interpret the dynamics as a
state-dependent switched system, formulate the latter as a hybrid
automaton as defined in~\cite{JL-KHJ-SNS-JZ-SSS:03}, and then employ
the invariance principle for hybrid automata to characterize its
asymptotic convergence properties.  However, this route is not valid
in general because one of the key assumptions required by the
invariance principle for hybrid automata is not satisfied by the
primal-dual dynamics. The next example justifies this claim.

\begin{example}\longthmtitle{The hybrid automaton corresponding to the
    primal-dual dynamics is not continuous}\label{ex:counter}
  {\rm Consider the concave optimization
    problem~\eqref{eq:concave-opt} on $\real$ with $f(x) = -(x-5)^2$
    and $g(x) = x^2-1$, whose set of primal-dual optimizers is $\sX
    \times \slm = \{(1,4)\}$. 
    The associated primal-dual dynamics takes the form
    \begin{subequations}\label{eq:example-dyn}
      \begin{align}
        \dot x & =-2(x-5) - 2x\lm,
        \\
        \dot \lm & = [x^2 - 1]_{\lm}^+. \label{eq:proj-part}
      \end{align}
    \end{subequations}
    We next formulate this dynamics as a hybrid automaton as defined
    in~\cite[Definition~II.1]{JL-KHJ-SNS-JZ-SSS:03}.  The idea to
    build the hybrid automaton is to divide the state space $\real
    \times \realnonnegative$ into two domains over which the vector
    field~\eqref{eq:example-dyn} is continuous. To this end, we define
    two modes represented by the discrete variable $q$, taking values
    in $\mathbf{Q} = \{1,2\}$.  The value $q = 1$ represents the mode
    where the projection in~\eqref{eq:proj-part} is active and $q = 2$
    represents the mode where it is not.  Formally, the projection is
    \emph{active} at $(x,\lm)$ if $ [g(x)]_{\lambda}^+ \neq g(x)$,
    i.e, $\lm = 0$ and $g(x) < 0$.
    The hybrid automaton is then given by the collection $H =
    (Q,X,f,\mathrm{Init},D,E,G,R)$, where $Q = \{q\}$ is the set of
    discrete variables, taking values in $\mathbf{Q}$; $X = \{x,\lm\}$
    is the set of continuous variables, taking values in $\mathbf{X} =
    \real \times \realnonnegative$; the vector field
    $\map{f}{\mathbf{Q} \times \mathbf{X}}{T \mathbf{X}}$ is defined
    by
    \begin{align*}
      f(1,(x,\lm)) & = \begin{bmatrix} -2(x-5)-2x\lm
        \\
        0 \end{bmatrix} ,
      \\
      f(2,(x,\lm)) & = \begin{bmatrix} -2(x-5)-2x\lm \\ x^2 - 1
      \end{bmatrix} ;
    \end{align*}
    $\mathrm{Init}= \mathbf{X}$ is the set of initial conditions;
    $\setmap{D}{\mathbf{Q}}{\mathbf{X}}$ specifies the domain of each
    discrete mode,
    \begin{align*}
      D(1) = (-1,1) \times \{0\}, \quad D(2) = \mathbf{X} \setminus
      D(1),
    \end{align*}
    i.e., the dynamics is defined by the vector field $(x,\lm) \to
    f(1,(x,\lm))$ over $D(1)$ and by $(x,\lm) \to f(2,(x,\lm))$ over
    $D(2)$; $E = \{(1,2),(2,1)\}$ is the set of edges specifying the
    transitions between modes; the guard map
    $\setmap{G}{\mathbf{Q}}{\mathbf{X}}$ specifies when a solution can
    jump from one mode to the other,
    \begin{align*}
      G(1,2) = \{(1,0), (-1,0)\}, \quad G(2,1) = (-1,1) \times \{0\} ,
    \end{align*}
    i.e., $ G(q,q')$ is the set of points where a solution jumps from
    mode $q$ to mode~$q'$; and, finally, the reset map
    $\setmap{R}{\mathbf{Q} \times \mathbf{X}}{\mathbf{X}}$ specifies
    that the state is preserved after a jump from one mode to another,
    \begin{align*}
      R( (1,2), (x,\lm) ) = R( (2,1), (x,\lm) ) = \{(x,\lm)\} .
    \end{align*}
    We are now ready to show that the hybrid automaton is not
    continuous in the sense defined by~\cite[Definition
    III.3]{JL-KHJ-SNS-JZ-SSS:03}. This notion plays a key role in the
    study of omega-limit sets and their stability, and is in fact a
    basic assumption of the invariance principle developed
    in~\cite[Theorem IV.1]{JL-KHJ-SNS-JZ-SSS:03}.  Roughly speaking,
    $H$ is continuous if two executions of $H$ starting close to one
    another remain close to one another. An execution of $H$ consists
    of a tuple $(\tau, q, x)$, where~$\tau$ is a hybrid time
    trajectory (a sequence of intervals specifying where mode
    transitions and continuous evolution take place), $q$ is a map
    that gives the discrete mode of the execution at each interval
    of~$\tau$, and $x$ is a set of differentiable maps that represent
    the evolution of the continuous state of the execution along
    intervals of~$\tau$. A necessary condition for two executions to
    ``remain close'' is to have the time instants of transitions in
    their mode for the executions (if there are any) close to one
    another.  To disprove the continuity of $H$, it is enough then to
    show that there exist two executions that start arbitrarily close
    and yet experience their first mode transitions at time instants
    that are not arbitrarily close.
    \begin{figure}
      \centering
      \includegraphics[width = 0.5 \linewidth]{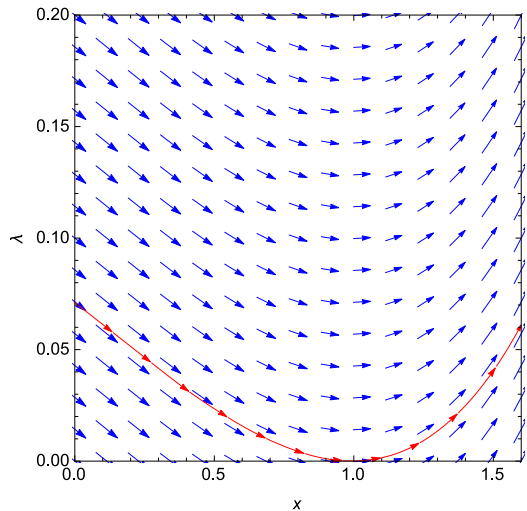}
      \caption{An illustration depicting the vector
        field~\eqref{eq:example-dyn} in the range $(x,\lm) \in [0,1.6]
        \times [0,0.2]$. As shown (with a red streamline), there
        exists a solution of~\eqref{eq:example-dyn} that starts at a
        point $(x(0),\lm(0))$ with $x(0) < 1$ and $\lm(0) > 0$ such
        that it remains in the domain $\lm > 0$ at all times except at
        one time instant $t'$ when $(x(t'),\lm(t')) =
        (1,0)$. }\label{fig:v-field}
    \end{figure}
    Select an initial condition $(x(0),\lm(0)) \in
    (0,1)\times(0,\infty)$ that gives rise to a solution
    of~\eqref{eq:example-dyn} that remains in the set $(0,1) \times
    (0,\infty)$ for a finite time interval $(0,t')$, $t' > 0$,
    satisfies $(x(t'),\lm(t')) = (1,0)$, and stays in the set
    $(1,\infty) \times (0,\infty)$ for some finite time interval
    $(t',T)$, $T > t'$.  The existence of such a solution becomes
    clear by plotting the vector field~\eqref{eq:example-dyn}, see
    Figure~\ref{fig:v-field}.  Note that by construction, this also
    corresponds to an execution of the hybrid automaton $H$ that
    starts and remains in domain $D(2)$ for the time interval $[0,T]$
    and so it does not encounter any jumps in its discrete mode.
    Specifically, for this execution, the hybrid time trajectory is
    the interval $[0,T]$, the discrete mode $q$ is always $2$ and the
    continuous state evolves as $t \mapsto (x(t),\lm(t))$. Further, by
    observing the vector field, we deduce that in every neighborhood
    of $(x(0),\lm(0))$, there exists a point
    $(\tilde{x}(0),\tilde{\lm}(0))$ such that a solution
    of~\eqref{eq:example-dyn} $t \mapsto
    (\tilde{x}(t),\tilde{\lm}(t))$ starting at
    $(\tilde{x}(0),\tilde{\lm}(0))$ reaches the set $(0,1) \times
    \{0\}$ in finite time $t_1 >0$, remains in $(0,1) \times \{0\}$
    for a finite time interval $[t_1,t_2]$, and then enters the set
    $(1,\infty) \times (0,\infty)$ upon reaching the point $(1,0)$.
    Indeed, this is true whenever $\tilde{x} < x(0)$ and $\tilde{\lm}
    < \lm(0)$. The execution of $H$ corresponding to this solution
    starts in $D(2)$, enters $D(1)$ in finite time $t_1$, and returns
    to $D(2)$ at time $t_2$. Specifically, the hybrid time trajectory
    consists of three intervals $\{[0,t_1],[t_1,t_2],[t_2,T']\}$,
    where we assume $T' > t_2$.  The discrete mode $q$ takes value $2$
    for the interval $[0,t_1]$, $1$ for the interval $[t_1,t_2]$, and
    $2$ for the interval $[t_2,T']$. The continuous state $t \mapsto
    (\tilde{x}(t),\tilde{\lm}(t))$ takes the same values as the
    solution of~\eqref{eq:example-dyn} explained above.  Thus, the
    value of the discrete variable representing the mode of the
    execution switches from $2$ to $1$ and back to $2$, whereas the
    execution corresponding to the solution of~\eqref{eq:example-dyn}
    starting at $(x(0),\lm(0))$ never switches mode. This shows that
    the hybrid automaton is not continuous. }  \oprocend
\end{example}

Interestingly, even though the hybrid automaton $H$ described in
Example~\ref{ex:counter} is not continuous, one can infer from
Figure~\ref{fig:v-field} that two solutions of~\eqref{eq:example-dyn}
remain close to each other if they start close enough. This suggests
that continuity with respect to the initial condition might hold
provided this notion is formalized the way it is done for traditional
nonlinear systems (and not as done for hybrid automata where both
discrete and continuous states have to be aligned). The next section
shows that this in fact is the case. This, along with the existence
and uniqueness of solutions, allows us to characterize the asymptotic
convergence properties of the primal-dual dynamics.

\myclearpage
\section{Convergence analysis of primal-dual
  dynamics}\label{sec:convergence}

In this section we show that the solutions of the primal-dual
dynamics~\eqref{eq:p-d-dynamics} asymptotically converge to a solution
of the constrained optimization problem~\eqref{eq:concave-opt}.  Our
proof strategy is to employ the invariance principle for Caratheodory
solutions of discontinuous dynamical systems stated in
Proposition~\ref{pr:invariance-cara}.  Our first step is then to
verify that all its hypotheses hold.

We start by stating a useful monotonicity property of the primal-dual
dynamics with respect to the set of primal-dual optimizers $\sX \times
\slm$. This property can be found in~\cite{KA-LH-HU:58,DF-FP:10} and
we include here its proof for completeness.

\begin{lemma}\longthmtitle{Monotonicity of the primal-dual dynamics
    with respect to primal-dual optimizers}\label{le:derivative-v}
  Let $(\xo,\lmo) \in \sX \times \slm$ and define $\map{V}{\real^n
    \times \real^m}{\realnonnegative}$,
  \begin{equation}\label{eq:v-func}
    V(x,\lm) = \frac{1}{2} \big(\norm{x - \xo}^2 + \norm{\lm - \lmo}^2
    \big).
  \end{equation}
  Then $\Lie_{\Xp} V(x,\lm) \le 0$ for all $(x,\lm) \in \DD$.
\end{lemma}
\begin{proof}
  By definition of $\Lie_{\Xp} V$ (cf. Section~\ref{subsec:disc}), we have
  \begin{align*}
    \Lie_{\Xp} V(x,\lm) & = (x - \xo)^\top \gradient_x L(x,\lm) + (\lm
    - \lmo)^\top [-\gradient_{\lm} L(x,\lm)]_{\lm}^+
    \\
    & = (x - \xo)^\top \gradient_x L(x,\lm) - (\lm - \lmo)^\top
    \gradient_{\lm} L(x,\lm)
    \\
    & \quad + (\lm - \lmo)^\top ([-\gradient_{\lm} L(x,\lm)]_{\lm}^+ +
    \gradient_{\lm} L(x,\lm)) .
  \end{align*}
  Since $L$ is concave in $x$ and convex in $\lm$, applying the first
  order condition of concavity and convexity for the first two terms
  of the above expression yields the following bound
  \begin{align*}
    \Lie_{\Xp} V(x,\lm) & \le L(x,\lm) - L(\xo,\lm) + L(x,\lmo) -
    L(x,\lm)
    \\
    & \quad + (\lm - \lmo)^\top ([-\gradient_{\lm} L(x,\lm)]_{\lm}^+ +
    \gradient_{\lm} L(x,\lm))
    \\
    & = L(\xo,\lmo) - L(\xo,\lm) + L(x,\lmo) - L(\xo,\lmo)
    \\
    & \quad + (\lm - \lmo)^\top ([-\gradient_{\lm} L(x,\lm)]_{\lm}^+ +
    \gradient_{\lm} L(x,\lm)) .
  \end{align*}
  Define the shorthand notation $M_1 = L(\xo,\lmo) - L(\xo,\lm)$, $M_2
  = L(x,\lmo) - L(\xo,\lmo)$, and $M_3 = (\lm - \lmo)^\top
  ([-\gradient_{\lm} L(x,\lm)]_{\lm}^+ + \gradient_{\lm} L(x,\lm))$,
  so that the above inequality reads
  \begin{align*}
    \Lie_{\Xp} V(x,\lm) \le M_1 + M_2 + M_3.
  \end{align*}
  Since $\lmo$ is a minimizer of the map $\lm \to L(\xo,\lm)$ over the
  domain $\realnonnegative^m$ and $\xo$ is a maximizer of the map $x
  \to L(x,\lmo)$, we obtain $M_1 , M_2 \le 0$. Replacing
  $-\gradient_{\lm} L(x,\lm) = g(x)$, one can write $M_3 =
  \sum_{i=1}^m T_i$, where for each $i$,
  \begin{align*}
    T_i = (\lm_i - (\lmo)_i)([g_i(x)]_{\lm_i}^+ - g_i(x)).
  \end{align*}
  If $\lm_i > 0$, then $[g_i(x)]_{\lm_i}^+ = g_i(x)$ and so $T_i = 0$.
  If $\lm_i = 0$, then $\lm_i - (\lmo)_i \le 0$ and
  $[g_i(x)]_{\lm_i}^+ - g_i(x) \ge 0$, which implies that $T_i \le 0$.
  Therefore, we get $M_3 \le 0$, and the result follows.  \qed
\end{proof}

Next, we show that the primal-dual dynamics can be written as a
projected dynamical system.

\begin{lemma}\longthmtitle{Primal-dual dynamics as a projected
    dynamical system}\label{le:primal-projected}
  The primal-dual dynamics can be written as a projected dynamical
  system.
\end{lemma}
\begin{proof}
  Consider the vector field $\map{X}{\real^n \times \real^m}{\real^n
    \times \real^m}$ defined by
  \begin{equation}\label{eq:grad-dynamics}
    X(x,\lm) = \begin{bmatrix} \gradient_x L(x,\lambda) 
      \\
      -\gradient_{\lambda} L(x,\lambda)
    \end{bmatrix}.
  \end{equation}
  We wish to show that $\Xp(x,\lm) = \Pi_{\DD}((x,\lm),X(x,\lm))$ for
  all $(x,\lm) \in \DD$.  To see this, note that the maps $\Xp$ and
  $X$ take the same values over
  $\intr(\DD) = \real^n \times \realpositive^m$.  Now consider any
  point $(x,\lm) \in \bd(\DD)$.  Let $I \subset \until{m}$ be the set
  of indices for which $\lm_i = 0$ and
  $(-\gradient_{\lm} L(x,\lm))_i < 0$. Then, there exist
  $\tilde{\delta} > 0$ such that, for all
  $\delta \in [0,\tilde{\delta})$ and for any $j \in \until{n+m}$, we
  have
  \begin{align*}
    (\proj_{\DD}( (x,\lm) + \delta X(x,\lm) ))_j =
    \begin{cases} 0, & \quad \text{ if } j-n \in I,
      \\
      (x,\lm)_j + \delta (X(x,\lm))_j, & \quad \text{ otherwise }.
    \end{cases}
  \end{align*}
  Consequently, using the definition of the projection operator, cf.
  Section~\ref{subsec:projected}, we get
  \begin{align*}
    (\Pi_{\DD}((x,\lm),X(x,\lm)))_j & = 
    \begin{cases}
      0, & \quad \text{ if } j-n \in I,
      \\
      (X(x,\lm))_j, & \quad \text{ otherwise } ,
    \end{cases}
  \end{align*}
  which implies $\Xp(x,\lm) = \Pi_{\DD}((x,\lm),X(x,\lm))$ for all
  $(x,\lm) \in \bd(\DD)$. This concludes the proof. \qed
\end{proof}

Next, we use Lemmas~\ref{le:derivative-v}
and~\ref{le:primal-projected} to show the existence, uniqueness, and
continuity of the solutions of $\Xp$ starting from $\DD$.  Our proof
strategy consists of using Lemma~\ref{le:primal-projected} and
Proposition~\ref{pr:existence-pds} to conclude the result. A minor
technical hurdle in this process is ensuring the Lipschitz property of
the vector field~\eqref{eq:grad-dynamics}, the projection of which on
$\DD$ is $\Xp$. We tackle this by using the monotonicity property of
the primal-dual dynamics stated in Lemma~\ref{le:derivative-v}
implying that a solution of $\Xp$ (if it exists) remains in a bounded
set, which we know explicitly. This further implies that, given a
starting point, there exists a bounded set such that the values of the
vector field outside this set do not affect the solution starting at
that point and hence, the vector field can be modified at the outside
points without loss of generality to obtain the Lipschitz property. We
make this construction explicit in the proof.

\begin{lemma}\longthmtitle{Existence, uniqueness, and continuity of 
    solutions of the primal-dual dynamics}\label{le:existence}
  Starting from any point $(x,\lm) \in \DD$, a unique solution $t
  \mapsto \gamma(t)$ of the primal-dual dynamics $\Xp$ exists and
  remains in $ (\DD) \, \cap \, V^{-1}(\le V(x,\lm))$.  Moreover, if a
  sequence of points $\{(x_k,\lm_k)\}_{k=1}^{\infty} \subset \DD$
  converge to $(x,\lm)$ as $k \to \infty$, then the sequence of
  solutions $\{t \mapsto \gamma_k(t)\}_{k=1}^{\infty}$ of $\Xp$
  starting at these points converge uniformly to the solution $t
  \mapsto \gamma(t)$ on every compact set of $[0,\infty)$.
\end{lemma}
\begin{proof}
  Consider $(x(0),\lm(0)) \in \DD$ and
  let $\eps > 0$. Define $V_0 = V(x(0),\lm(0))$, where $V$ is given
  in~\eqref{eq:v-func}, and let $\WW_{\eps} = V^{-1}(\le V_0 +
  \epsilon)$. Note that $\WW_{\eps}$ is convex, compact, and
  $V^{-1}(\le V_0) \subset \intr(\WW_{\eps})$.  Let
  $\map{X^{\WW_{\eps}}}{\real^n \times \real^m} {\real^n \times
    \real^m}$ be a vector field defined as follows: equal to $X$ on
  $\WW_{\eps}$ and, for any $(x,\lm) \in (\real^n \times \real^m)
  \setminus \WW_{\eps}$,
  \begin{align*}
    X^{\WW_{\eps}}(x,\lm) = X(\proj_{\WW_{\eps}}(x, \lm)).
  \end{align*} 
  The vector field $X^{\WW_{\eps}}$ is Lipschitz on the domain
  $\real^n \times \real^m$.  To see this, note that $X$ is Lipschitz
  on the compact set $\WW_{\eps}$ with some Lipschitz constant $K > 0$
  because $f$ and $g$ have locally Lipschitz gradients.  Let
  $(x_1,\lm_1), (x_2,\lm_2) \in \real^n \times \real^m$. Then,
  \begin{align*}
    \norm{X^{\WW_{\eps}}(x_1,\lm_1) - X^{\WW_{\eps}}(x_2,\lm_2)} & =
    \norm{X(\proj_{\WW_{\eps}}(x_1,\lm_1)) -
      X(\proj_{\WW_{\eps}}(x_2,\lm_2))}
    \\
    & \le K \norm{\proj_{\WW_{\eps}}(x_1,\lm_1) -
      \proj_{\WW_{\eps}}(x_2, \lm_2)}
    \\
    & \le K \norm{(x_1,\lm_1) - (x_2,\lm_2)}.
  \end{align*}
  The last inequality follows from the Lipschitz property of the
  map $\proj_{\WW_{\eps}}$ (cf. Section~\ref{subsec:projected}). 

  Next, we employ Proposition~\ref{pr:existence-pds} to establish the
  existence, uniqueness, and continuity with respect to the initial
  condition of the solutions of the projected dynamical system,
  $\Xp^{\WW_{\eps}}$, associated with $X^{\WW_{\eps}}$ and $\DD$.  Our
  proof then concludes by showing that in fact all solutions of the
  projected system $\Xp^{\WW_{\eps}}$ starting in $\WW_\eps \cap \DD$
  are in one-to-one correspondence with the solutions of $\Xp$
  starting in $\WW_\eps \cap \DD$.  Let
  $\map{\Xp^{\WW_{\eps}}}{\DD}{\real^n \times \real^m}$ be the map
  obtained by projecting $X^{\WW_{\eps}}$ with respect to $\DD$,
  \begin{align*}
    \Xp^{\WW_{\eps}}(x,\lm) = \Pi_{\DD}( (x,\lm),X^{\WW_{\eps}}(x,\lm)
    ),
  \end{align*}
  for all $(x,\lm) \in \DD$. Since $\Xp$ is the projection of $X$ with
  respect to $\DD$, we deduce that $\Xp^{\WW_{\eps}} = \Xp$ over the
  set $\WW_\eps \cap \DD$.  Since $X^{\WW_{\eps}}$ is Lipschitz,
  following Proposition~\ref{pr:existence-pds}, we obtain that
  starting from any point in $\DD$, a unique solution of
  $\Xp^{\WW_{\eps}}$ exists over $[0,\infty)$ and is continuous with
  respect to the initial condition. Consider any solution $t \mapsto
  (\xt(t),\lmt(t))$ of $\Xp^{\WW_\eps}$ that starts in $\WW_\eps \cap
  \DD$.  Note that since the solution is absolutely continuous and $V$
  is continuously differentiable, the map $t \mapsto
  V(\xt(t),\lmt(t))$ is differentiable almost everywhere on
  $[0,\infty)$, and hence
  \begin{align*}
    \frac{d}{dt} V(\xt(t),\lmt(t)) = \Lie_{\Xp^{\WW_{\eps}}}
    V(\xt(t),\lmt(t)) ,
  \end{align*}
  almost everywhere on $[0,\infty)$. From Lemma~\ref{le:derivative-v}
  and the fact that $\Lie_{\Xp^{\WW_{\eps}}} V$ and $\Lie_{\Xp} V$ are
  the same over $\WW_\eps \cap \DD$, we conclude that $V$ is
  non-increasing along the solution. This means the solution remains
  in the set $\WW_\eps \cap \DD$. Finally, since $\Xp^{\WW_{\eps}}$
  and $\Xp$ are same on $\WW_\eps \cap \DD$, we conclude that $t
  \mapsto (\xt(t),\lmt(t))$ is also a solution of $\Xp$. Therefore,
  starting at any point in $\WW_\eps \cap \DD$, a solution of $\Xp$
  exists. Using Lemma~\ref{le:derivative-v}, one can show that, if a
  solution of $\Xp$ that starts from a point in $\WW_\eps \cap \DD$
  exists, then it remains in $\WW_\eps \cap \DD$ and so is a solution
  of $\Xp^{\WW_\eps}$. This, combined with the uniqueness of solutions
  of $\Xp^{\WW_\eps}$, implies that a unique solution of $\Xp$ exists
  starting from any point in $\WW_\eps \cap \DD$. In particular, this
  is true for the point $(x(0),\lm(0))$.  Finally, from the continuity
  of solutions of $\Xp^{\WW_\eps}$ and the one-to-one correspondence
  of solutions of $\Xp$ and $\Xp^{\WW_\eps}$ starting $\WW_\eps \cap
  \DD$, we conclude the continuity with respect to initial condition
  for solutions of $\Xp$ starting in $V^{-1}(x(0),\lm(0))$. Since
  $(x(0),\lm(0))$ is arbitrary, the result follows.
  \qed
\end{proof}

The next result states the invariance of the omega-limit set of any
solution of the primal-dual dynamics. This ensures that all hypotheses
of the invariance principle for Caratheodory solutions of
discontinuous dynamical systems, cf.
Proposition~\ref{pr:invariance-cara}, are satisfied. 

\begin{lemma}\longthmtitle{Omega-limit set of solution of primal-dual
    dynamics is invariant}\label{le:omega-inv}
  The omega-limit set of any solution of the primal-dual dynamics
  starting from any point in $\real^n \times \realnonnegative^m$ is
  invariant under~\eqref{eq:p-d-dynamics}.
\end{lemma}

The proof of Lemma~\ref{le:omega-inv} follows the same line of
argumentation that the proof of invariance of omega-limit sets of
solutions of locally Lipschitz vector fields, cf. ~\cite[Lemma
4.1]{HKK:02}. 
We are now ready to establish our main result, the asymptotic
convergence of the solutions of the primal-dual dynamics to a solution
of the constrained optimization problem.

\begin{theorem}\longthmtitle{Convergence of the primal-dual dynamics
    to a primal-dual optimizer}\label{th:convergence}
  The set of primal-dual solutions of~\eqref{eq:concave-opt} is
  globally asymptotically stable on $\real^n \times
  \realnonnegative^m$ under the primal-dual
  dynamics~\eqref{eq:p-d-dynamics}, and the convergence of each
  solution is to a point.
\end{theorem}
\begin{proof}
  Let $(\xo,\lmo) \in \sX \times \slm$ and consider the function~$V$
  defined in~\eqref{eq:v-func}.  For $\delta > 0$, consider the
  compact set $\SS = V^{-1}(\le \delta) \cap (\real^n \times
  \realnonnegative^m)$.  From Lemma~\ref{le:existence}, we deduce that
  a unique solution of $\Xp$ exists starting from any point in $\SS$,
  which remains in~$\SS$.  Moreover, from Lemma~\ref{le:omega-inv},
  the omega-limit set of each solution starting from any point in
  $\SS$ is invariant. Finally, from Lemma~\ref{le:derivative-v},
  $\Lie_{\Xp} V(x,\lm) \le 0$ for all $(x,\lm) \in \SS$. Therefore,
  Proposition~\ref{pr:invariance-cara} implies that any solution of
  $\Xp$ staring in $\SS$ converges to the largest invariant set $M$
  contained in $\cl(Z)$, where $Z = \setdef{(x,\lm) \in
    \SS}{\Lie_{\Xp} V(x,\lm) = 0}$.  From the proof of
  Lemma~\ref{le:derivative-v}, $\Lie_{\Xp} V(x,\lm) = 0$ implies
  \begin{align*}
    L(\xo,\lmo) - L(\xo,\lm) &= 0,
    \\
    L(x,\lmo) - L(\xo,\lmo) & = 0,
    \\
    (\lm_i - (\lmo)_i)([g_i(x)]_{\lm_i}^+ - g_i(x)) & = 0 , \quad
    \text{for all } i \in \until{m}.
  \end{align*}
  Since $f$ is strictly concave, so is the function $x \mapsto
  L(x,\lmo)$ and thus $L(x,\lmo) = L(\xo,\lmo)$ implies $x = \xo$.
  The equality $L(\xo,\lmo) - L(\xo,\lm) = 0$ implies $\lm^\top g(\xo)
  = 0$. Therefore $Z = \setdef{(x,\lm) \in \SS}{x = \xo, \lm^\top
    g(\xo) = 0}$ is closed. 
  Let $(\xo,\lm) \in M \subset Z$.  The solution of $\Xp$ starting at
  $(\xo,\lm)$ remains in $M$ (and hence in $Z$) only if $\gradient
  f(\xo) - \sum_{i=1}^m \lm_i \gradient g_i(\xo) = 0$. This implies
  that $(\xo,\lm)$ satisfies the KKT conditions~\eqref{eq:KKT} and
  hence, $M \subset \sX \times \slm$.  Since the initial choice
  $\delta > 0$ is arbitrary, we conclude that the set $\sX \times
  \slm$ is globally asymptotically stable on $\real^n \times
  \realnonnegative^m$.  Finally, we note that convergence is to a
  point in $\sX \times \slm$.  This is equivalent to saying that the
  omega-limit set $\Omega(x,\lm) \subset \sX \times \slm$ of any
  solution $t \mapsto (x(t),\lm(t))$ of $\Xp$ is a singleton. This
  fact follows from the definition of omega-limit set and the fact
  that, by Lemma~\ref{le:derivative-v}, primal-dual optimizers are
  Lyapunov stable.
  This concludes the proof.  \qed
\end{proof}

\begin{remark}\longthmtitle{Alternative proof
    strategy via evolution variational inequalities}\label{re:alternative}
  {\rm We briefly describe here an alternative proof strategy to the
    one we have used here to establish the asymptotic convergence of
    the primal-dual dynamics.  The Caratheodory solutions of the
    primal-dual dynamics can also be seen as solutions of an evolution
    variational inequality (EVI) problem~\cite{BB-DG:05}. Then, one
    can show that the resulting EVI problem has a unique solution
    starting from each point in $\DD$, which moreover remains in
    $\DD$. With this in place, the LaSalle Invariance
    Principle~\cite[Theorem 4]{BB-DG:05} for the solutions of the EVI
    problem can be applied to conclude the convergence to the set of
    primal-dual optimizers. } \oprocend
\end{remark}

\begin{remark}\longthmtitle{Primal-dual dynamics with
    gains}\label{re:c-dynamics}
  {\rm In power network optimization
    problems~\cite{CZ-UT-NL-SL:14,EM-CZ-SL:14,XZ-AP:14} and network
    congestion control problems~\cite{JTW-MA:04,SHL-FP-JCD:02},
    it is common
    to see generalizations of the primal-dual dynamics involving gain
    matrices.  Formally, these dynamics take the form
    \begin{subequations}\label{eq:p-d-dynamics-gains}
      \begin{align}
        \dot x & = K_1 \gradient_x L(x,\lambda) ,
        \\
        \dot \lambda & = K_2 [-\gradient_{\lambda}
        L(x,\lambda)]_{\lambda}^+,
      \end{align}
    \end{subequations}
    where $K_1 \in \real^{n \times n}$ and $K_2 \in \real^{m \times
      m}$ are diagonal, positive definite matrices.  In such cases,
    the analysis performed here can be replicated following the same
    steps but using instead the Lyapunov function
    \begin{align*}
      V'(x,\lm) = \frac{1}{2} ((x-\xo)^\top K_1^{-1} (x-\xo)
      +(\lm-\lmo)^\top K_2^{-1} (\lm-\lmo)) ,
    \end{align*}
    to establish the required monotonicity and convergence properties
    of~\eqref{eq:p-d-dynamics-gains}.}\oprocend
\end{remark}

\myclearpage
\section{Conclusions}\label{sec:conclusions}

We have considered the primal-dual dynamics for a constrained
concave optimization problem and established the asymptotic
convergence of its Caratheodory solutions to a primal-dual optimizer
using classical notions from stability theory.  Our technical approach
has employed results from projected dynamical systems to establish
existence, uniqueness, and continuity of the solutions, and the
invariance principle for discontinuous Caratheodory systems to
characterize their asymptotic convergence.  We have also shown by
means of a counterexample how a proof strategy based on interpreting
the primal-dual dynamics as a hybrid automaton is not valid in general
because of the lack of continuity (understood in the hybrid sense) of
the solutions.  The technical approach presented in the paper opens up
the possibility of rigorously characterizing the robustness properties
of the primal-dual dynamics against unmodeled dynamics, disturbances,
and noise. Motivated by applications to power networks, we also plan
to explore the design of discontinuous dynamics that can find the
solutions to semidefinite programs and quadratically constrained
quadratic programs.

\section{Acknowledgements}
The first and the third author wish to thank Dr. Bahman Gharesifard
and Dr. Dean Richert for fruitful discussions on the primal-dual
dynamics.  This research was partially supported by NSF Award
ECCS-1307176, Los Alamos National Lab through a DoE grant, and DTRA
under grant 11376437.

\setlength{\bibsep}{1.1ex}
{\small

}

\end{document}